\documentclass[12pt]{article}
\usepackage{mathrsfs}
\usepackage{amsthm}
\usepackage{amssymb}
\usepackage{latexsym}
\usepackage{amsmath,amsfonts}
\usepackage{mathrsfs}
\usepackage{cases}
\usepackage{latexsym,bm}
\usepackage{indentfirst}
\usepackage{xcolor}
\usepackage{ifpdf}
\usepackage{graphicx}
\usepackage{epstopdf}
\usepackage{epsfig}
\usepackage{psfrag}
\usepackage{enumitem}
\usepackage{epstopdf}
\usepackage{verbatim}
\usepackage{color}

\usepackage[
pdfauthor={Lu},
pdftitle={Planar wheel-like bricks},
pdfstartview=XYZ,
bookmarks=true,
colorlinks=true,
linkcolor=blue,
urlcolor=blue,
citecolor=blue,
bookmarks=true,
linktocpage=true,
hyperindex=true
]{hyperref}

\title{ Planar wheel-like bricks
\footnote{ The research
 is partially supported by NSFC (No. 12271235) and NSF of Fujian Province (No. 2021J06029).
\newline E-mail addresses: flianglu@163.com (F. Lu), xuejinxin00@163.com (J. Xue)}}
\author{Fuliang Lu, Jinxin Xue \\
\small {School of Mathematics and Statistics, Minnan Normal University, Zhangzhou,  China}}

\date{}

\newtheorem{lem}{Lemma}[section]
\newtheorem{thm}[lem]{Theorem}

\newtheorem{pro}[lem]{Proposition}
\newtheorem{conj}[lem]{Conjecture}

\def\typeone{type~\uppercase\expandafter{\romannumeral 1}}
\def\typetwo{type~\uppercase\expandafter{\romannumeral 2}}

\begin{document}
\newcommand{\udots}{\mathinner{\mskip1mu\raise1pt\vbox{\kern7pt\hbox{.}}
\mskip2mu\raise4pt\hbox{.}\mskip2mu\raise7pt\hbox{.}\mskip1mu}}
\maketitle
\begin{abstract}

    An edge $e$ in a matching covered graph $G$ is {\em removable} if $G-e$ is matching covered;
     a pair $\{e,f\}$ of edges of  $G$ is a \emph{removable doubleton} if $G-e-f$ is matching covered, but neither $G-e$ nor $G-f$ is. Removable edges and removable doubletons are called {\em removable classes},
     which was introduced by Lov\'asz and Plummer  in connection with ear decompositions of matching covered graphs.
    A \emph{brick} is a nonbipartite matching covered graph without nontrivial tight cuts.
    A brick $G$ is \emph{wheel-like} if $G$ has a vertex $h$,  such that every removable class of $G$ has an edge incident with $h$. Lucchesi and  Murty conjectured that  every planar wheel-like brick is an odd wheel.
We present a proof of this conjecture in this paper.

\end{abstract}
\vskip 0.2in \baselineskip 0.1in
\section{Introduction}

Graphs considered in this paper are finite and loopless (multiple edges are allowed, where an edge $e$ of a graph $G$ is a
{\em multiple edge} if there are at least two edges of $G$ with the same ends as $e$). We follow \cite{BM08} and  \cite{LP86} for undefined notations and terminologies.
Let $G$ be a graph with the vertex set $V(G)$ and the edge set $E(G)$.
For $X,Y\subseteq V(G)$, by $E_G[X,Y]$, or simply $E[X,Y]$,  we mean the set of edges of $G$ with one end vertex in $X$ and the other end vertex in $Y$.
We shall simply write $E[x,y]$ for $E[X,Y]$ if $X=\{x\}$ and $Y=\{y\}$.
Let $\partial_G(X)=E_G[X,\overline{X}]$ be an edge cut of $G$, where $\overline{X}=V(G)\setminus X$. (If $G$ is understood,
the subscript $G$ is omitted.)
If $X=\{u\}$, then we denote $\partial_G(\{u\})$, for brevity, by $\partial_G(u)$ or $\partial(u)$.
An edge cut $\partial(X)$ is \emph{trivial} if $|X|=1$ or $|\overline{X}|=1$.

Let $G$ be a graph with a perfect matching.
An edge $e$ in $G$ is \emph{forbidden} if $e$ does not lie in any perfect matching of $G$.
A connected nontrivial graph is \textit{matching covered} if each of its edges is not forbidden.
We denote by $G/(X\rightarrow x)$ the graph obtained from $G$
by contracting $X$ to a single vertex $x$, for brevity, by $G/X$.
The graphs
$G/X$ and $G/\overline{X}$ are the two \emph{$\partial(X)$-contractions} of $G$.
An edge cut $\partial(X)$ is \textit{separating} if both $\partial(X)$-contractions of $G$ are matching covered, and is \textit{tight} if every perfect matching contains exactly one edge of $\partial(X)$.
A matching covered nonbipartite graph is a \textit{brick} if every tight cut is trivial, is \emph{solid} if every separating cut is a tight cut.
Bricks   are 3-connected and bicritical \cite{ELP82},  where a nontrivial graph $G$ is {\em bicritical} if $G -\{u, v\}$ has a perfect matching for any two vertices $u$ and $v$ of $G$.

Lov\'asz~\cite{Lovasz87} proved
that any matching covered graph can be decomposed
into a unique list of bricks and braces (a matching covered bipartite graph in which every tight cut is trivial) by a
procedure called the tight cut decomposition. In particular, any two decompositions of a matching covered graph $G$ yield the same number of bricks; this number is denoted by $b(G)$.
A \emph{near-brick} $G$ is a matching covered graph with $b(G)=1$. Obviously, a brick is a near-brick.

We say that an edge $e$ in a matching covered graph $G$ is \textit{removable} if $G-e$ is matching covered.
A pair $\{e,f\}$ of edges of a matching covered graph $G$ is a \emph{removable doubleton} if $G-e-f$ is matching covered, but neither $G-e$ nor $G-f$ is. Removable edges and removable doubletons are called {\em removable classes}.

 Lov\'asz \cite{lo} proved  that every brick distinct from $K_4$ and
$\overline{C_6}$ (see Figure \ref{fig:H8c} (left)) has a removable edge. Improving Lov\'asz's result,  Carvalho et al.  obtained  a lower bound of removable classes of
a brick  in terms of  the maximum degree.

 \begin{thm}[\cite{CLM02c}] \label{thm:re_in_brick}
    Every brick has at least $\Delta(G)$ removable classes, where $\Delta(G)$
    is the maximum of the degrees of vertices in $G$.
\end{thm}

 In \cite{zl} and \cite{zg}, Zhai et al., and Zhai and Guo presented  lower bounds of the number of removable edges (ears)   in
a matching covered graph $G$ in terms of the number of the perfect matchings needed to cover all edges of $G$ and the edge-chromatic number, respectively.
 Wu  et al. \cite{wl} showed every claw-free brick $G$ with more
than 6 vertices has at least $5|V(G)|/8$ removable edges.
   Wang   et al. \cite{wdl} obtained the number of removable edges of Halin graphs $G$ with
even number of vertices, other than $K_4$,
$\overline{C_6}$  and $R_8$ (see Figure \ref{fig:H8c} (right)), is at least $(|V(G)|+2)/4$.
 Generally, Lucchesi and Murty \cite{Lucchesi2024} conjectured that there exist a positive real number $c$ and an integer $N$ such that any brick $G$ of
order  at least $N$ has $c|V(G)|$ removable edges.

For an integer $k\geq3$, the \emph{wheel} $W_k$ is the graph obtained from a cycle $C$ of length $k$ by adding a new vertex $h$ and joining it to all vertices of $C$.
The cycle $C$ is the \emph{rim} of $W_k$, the vertex $h$ is its \emph{hub}. Obviously, every wheel is planar. A wheel $W_k$ is odd if $k$ is odd. The graph $K_4$  (a complete graph with 4 vertices)  is an odd wheel that every edge lies in a removable doubleton. For an odd wheel other than $K_4$,
it can be checked every edge on the rim is not removable, and  every edge incident with the hub is removable  (for example, see Exercise 2.2.4 in \cite{Lucchesi2024}). More generally, we say
a brick $G$ is \emph{wheel-like} if $G$ has a vertex $h$, called its \emph{hub}, such that every removable class of $G$ has an edge in $\partial(h)$. Lucchesi and Murty  made the following conjecture.
\begin{conj}[\cite{Lucchesi2024}]\label{main}
    Every planar wheel-like brick is an odd wheel.
\end{conj}

We present a proof of this conjecture. The following is the main result.
\begin{thm}\label{main}
    Every planar wheel-like brick $G$ is an odd wheel and all  multiple edges of $G$ are incident with the hub.
\end{thm}

\section{ Preliminaries }
We begin with some notations.
For a vertex set $X\subset V(G)$, denote by $G[X]$ the subgraph induced by $X$, by $N(X)$, or simply $N(u)$ when $X=\{u\}$, the set of all vertices in $\overline{X}$ adjacent to vertices in $X$.
Let $d_G(u)=|N(u)|$.
Let $G$ be a graph with a perfect matching.
A nonempty vertex set $S$ of $G$ is a {\em barrier} of $G$ if $o(G-S)=|S|$, where $o(G-S)$ is the number of components with odd number of vertices of $G-S$. Tutte  proved the following theorem which characterizes
graphs that have a perfect matching in 1947.

\begin{thm}[Tutte]\label{thm:Tutte}
    A graph $G$ has a perfect matching if and only if $o(G-X)\le|X|$, for every $X\subseteq V(G)$.
\end{thm}
The following result can be obtained by Theorem \ref{thm:Tutte}.
\begin{lem}[\cite{CLM12}]\label{forbidden}
    Assume that $G$ is a graph with a perfect matching. An edge $uv$ is forbidden if and only if there exists a barrier containing $u$ and $v$.
\end{lem}

Let $G$ and $H$ be two vertex-disjoint graphs and let $u$ and $v$ be vertices of $G$ and $H$, respectively, such that $d_G(u)=d_H(v)$.
Moreover, let $\theta$ be a given bijection between $\partial_G(u)$ and $\partial_H(v)$.
We denote by $(G(u)\odot H(v))_\theta$ the graph obtained from the union of $G-u$ and $H-v$ by joining, for each edge $e$ in $\partial_H(v)$, the end of $e$ in $H$ belonging to $V(H)-v$ to the end of $\theta(e)$ in $G$ belonging to $V(G)-u$;
and refer to $(G(u)\odot H(v))_\theta$ as the graph obtained by \emph{splicing $G$ (at $u$), with $H$ (at $v$), with respect to the bijection $\theta$}, for brevity, to $G(u)\odot H(v)$.
In general, the graph resulting from splicing two graphs $G$ and $H$ depends on the choice of $u$, $v$ and $\theta$. The following proposition can be gotten by the definition of matching covered graphs directly (see Theorem 2.13 in \cite{Lucchesi2024} for example).

\begin{pro}\label{thm:MC_IS_MC}
    The splicing of two matching covered graphs is also matching covered.
\end{pro}

   By Proposition \ref{thm:MC_IS_MC}, $G(u)\odot H(v)$ is a matching covered graph if $G$ and $H$ are matching covered. Let $C=\partial (V(G)\setminus \{u\})$.
   Then the two $C$-contractions of $G(u)\odot H(v)$ are $G$ and $H$, respectively.
   Therefore, the edge cut $C$ is separating in $G(u)\odot H(v)$.
   A matching covered graph $G$ is \emph{near-bipartite} if it has a removable doubleton $R$ such that the subgraph $G-R$ obtained by the deletion of $R$ is a (matching covered) bipartite graph. In fact,     every brick with a removable doubleton is near-bipartite \cite{Lovasz87}.  So, if $\{e, f\}$ is a removable doubleton of a brick $G$, then both ends of $e$ lie in one color class of $G-\{e, f\}$, and both ends of $f$ lie in the other color class.
    Clearly, any two pairs of removable classes in a brick are disjoint.
The following  theorem is about  removable edges of a near-bipartite brick.

\begin{thm}[Theorem 9.17 in \cite{Lucchesi2024}]\label{thm:near-bi_non-adj}
    Every simple near-bipartite brick, distinct from $K_4$, $\overline{C_6}$ and  $R_8$,  has two nonadjacent removable edges.
\end{thm}

Let $G$ be a matching covered graph.
A separating cut $C$ of $G$ is a \emph{robust} cut if $C$ is not tight and both $C$-contractions of $G$ are near-bricks.

\begin{thm}[\cite{CLM02II}]\label{exist-robust}
    Every nonsolid brick $G$ has a robust cut.
\end{thm}

\begin{thm}[\cite{CLM02II}]\label{robust_H}
    Let $G$ be a brick and $\partial(X)$ be a robust cut of $G$.
    Then, there exists a subset $X'$ of $X$ and a superset $X''$ of $X$ such that  $G/\overline{X'}$ and $G/{X''}$  are bricks and the graph $H$, obtained from $G$ by contracting $X'$ and $\overline{X''}$ to single vertices $x'$ and $\overline{x''}$, respectively, is bipartite and matching covered, where $x'$ and $\overline{x''}$ lie in different color classes of $H$.
\end{thm}
As a type of (solid) bricks,
odd wheels play a crucial  role in our proof.
\begin{pro}[\cite{CLM06}]\label{thm:sim-six-w5}
    Let $G$ be a simple brick on six vertices.
    Then $G$ is either nonsolid or  the 5-wheel $W_5$.
\end{pro}
\begin{thm}[\cite{CLM06}]\label{thm:planar_solid}
    Every simple planar solid brick $G$ is an odd wheel.
\end{thm}

 Let $\partial (X)$ be a separating cut of  matching covered graph $G$  and $e\in E(G)$.  If $e\notin E(G/X)$, we also say $e$ is removable in $G/X$. The following results are about the removable edges in $C$-contractions.
\begin{lem}[Corollary 8.9 in \cite{Lucchesi2024}]\label{thm:re_also_re}
    Let $G$ be a matching covered graph, and let $C$ be a separating cut of $G$.
    If an edge $e$ is removable in both $C$-contractions of $G$, then  $e$ is removable in $G$.
\end{lem}

\begin{thm}[Lemma 3.1 in \cite{CLM02II}]\label{thm:removable doubleton}
    Let $C=\partial(X)$ be a separating but not tight cut of a matching covered graph $G$ and let $H=G/\overline{X}$.
    Assume that $H$ is a brick, and let $R$ be a removable doubleton of $H$.
    If $R\cap C=\emptyset$ or if the edge of $R\cap C$ is removable in $G/X$ then $R-C$ contains an edge which is removable in $G$.
\end{thm}

\begin{lem}\label{lem:removable doubleton}
    Assume that $G$ is a brick, and  $\partial(X)$ is an edge cut of $G$ such that both $G/(X\rightarrow x)$ and $G/(\overline{X}\rightarrow \overline{x})$
    are brick. If $e\in \partial(X)$, and $\{e,f\}$ and $\{e,g\}$ are removable doubletons of $G/X$ and $G/\overline{X}$, respectively, then  $\{f,g\}$ is a removable doubleton of $G$.
    \end{lem}
    \begin{proof} Obviously, $d_{G/X}(x)=d_{G/\overline{X}}(\overline{x}) $. Let $H_1= G/X-\{e,f\}$ and $H_2=G/\overline{X}-\{e,g\}$. As $e\in \partial(X)$,   $d_{H_1}(x)=d_{H_2}(\overline{x}) $. Then $G-\{e,f,g\}$ is isomorphic to $H_1(x) \odot H_2(\overline{x})$.
    By the definition of removable doubletons, $H_1$ and $H_2$ are matching covered bipartite graphs. Therefore, $G-\{e,f,g\}$  is matching covered by Proposition \ref{thm:MC_IS_MC}.
    Assume that $A_i$ and $B_i$ are the two color classes of $H_i$ such that both ends of $e$ lie in $B_i$ for $i=1,2$. Then $G-\{e,f,g\}$  is a bipartite graph with two color classes: $A_1\cup B_2\setminus\{\overline{x}\}$ and $A_2\cup B_1\setminus\{{x}\}$. Note that $E_G[B_1, B_2]=\{e\}$. So $G-\{f,g\}$ is bipartite. And then it is matching covered by Theorem 4.1.1 in \cite{LP86}. As  two ends of $f$  lie in $A_1$, and two ends of $g$  lie in $A_2$,
      $G-f$ and $G-g$ are not matching covered by a simple computation. Therefore, the result follows.
    \end{proof}

For nonremovable edges in a bipartite graph, we have the following result.

\begin{lem}[{Lemma 8.2 in \cite{Lucchesi2024}}]
\label{lem:nonre-bi}
Let $G$ be a bipartite matching covered graph with $A$ and $B$ as its color classes, and $|E(G)|\ge2$.
An edge $uv$ of $G$, with $u\in A$ and $v\in B$, is not removable in $G$ if and only if there exist nonempty proper subsets $A_1$ and $B_1$ of $A$ and $B$, respectively, such that:\\

{\rm 1)} the subgraph $G[A_1\cup B_1]$  is matching covered, and

{\rm 2)} $u\in A_1$ and $v\in B\setminus B_1$, and $E[A_1,B\setminus B_1 ]=\{uv\}$.
\end{lem}

\section{Lemmas}
In this section, we will present some  useful lemmas.
\begin{pro}\label{pro:uvnonremovable}
    Let $G$ be a matching covered graph.
    Assume that $N(u)=\{u_1,u_2,u_3\}$ and $G[\{u,u_2,u_3\}]$ is a triangle.
    Then $uu_1$ is not removable.
\end{pro}
\begin{proof}
It can be checked that $\{u_2,u_3\}$ is a barrier of $G-uu_1$.
    By Lemma \ref{forbidden}, $u_2u_3$ is forbidden in $G-uu_1$.
    So the result holds.
\end{proof}
We say an edge $e$ in a matching covered graph satisfies the {\em triangle condition} if one end of $e$ is of degree 3, and lie in a triangle that does not contain $e$. If $e$ satisfies the triangle condition, then $e$ is not removable by  Proposition \ref{pro:uvnonremovable}.

\begin{pro}\label{pro:H8_not_wheel-li}
    $\overline{C_6}$ and $R_8$ are not
  wheel-like.
\end{pro}
\begin{proof} As shown in  Figure \ref{fig:H8c}, $\{e_i,f_i\}$ is a removable doubleton  for $i=1,2,3$, and the bold edge
     is a removable edge. So the result follows.
 \end{proof}
\begin{figure}[!h]
    \centering
    \begin{minipage}[t]{0.48\textwidth}
    \centering
    \includegraphics[totalheight=2cm]{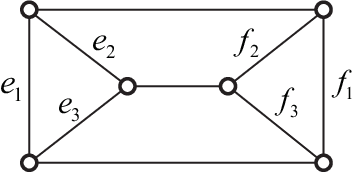}
    \end{minipage}
    \begin{minipage}[t]{0.48\textwidth}
    \centering
    \includegraphics[totalheight=1.9cm]{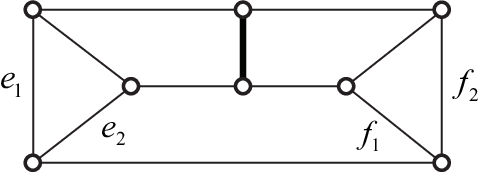}
    \end{minipage}
      \caption{$\overline{C_6}$ (left) and $R_8$ (right).}
      \label{fig:H8c}
\end{figure}

\begin{lem}\label{lem:ne-bi_k4}
    Let $G$ be a simple near-bipartite brick.
    Then $G$ is wheel-like if and only if $G$ is isomorphic to $K_4$.
\end{lem}
\begin{proof}
    If $|V(G)|=4$, then $G$ is isomorphic to $K_4$, which is wheel-like.
   Suppose that $|V(G)|\geq6$.    By  Proposition \ref{pro:H8_not_wheel-li}$, \overline{C_6}$ and $R_8$ are not wheel-like.
    If $G$ is distinct from $\overline{C_6}$ and  $R_8$, then $G$ has two nonadjacent removable edges by Theorem \ref{thm:near-bi_non-adj}, so it is not wheel-like.
\end{proof}

We will need the following theorem about planar graphs.

\begin{thm}[Kuratowski's Theorem]\label{thm:plaran_K3,3}
    Let $G$ be a graph. Then $G$ is nonplanar if and only if $G$ contains
    a subgraph that is a subdivision of either $K_{3,3}$ or $K_5$.
\end{thm}

\begin{lem}\label{lem:planarG/X}
    Let $G$ be a planar brick and $\partial(X)$ be a separating cut of $G$.
    Then $G/X$ and $G/\overline{X}$ are planar.
\end{lem}
\begin{proof}
    By Theorem \ref{thm:plaran_K3,3}, $G$ contains no
     subgraphs that is a subdivision of either $K_{3,3}$ or $K_5$. As $G/{\overline{X}}$ is matching covered, it is 2-connected (see Remark 2 on Page 145 of \cite{LP86}). Then $G[{X}]$ is connected.
    So $G/X$ contains no subgraphs that is a subdivision of either $K_{3,3}$ or $K_5$.
    Therefore,  $G/X$  is planar. So does  $G/\overline{X}$.
\end{proof}

\begin{pro}\label{lem:six_vertices_W5}
    Let $G$ be a simple planar brick with six vertices.
    Then $G$ is a wheel-like brick if and only if $G$ is $W_5$.
\end{pro}
\begin{proof}
    By Proposition \ref{thm:sim-six-w5}, $G$ is either nonsolid or  $W_5$. As $W_5$  is wheel-like, we may  assume that $G$ is a nonsolid brick.
    Then $G$ has a nontrivial separating cut $C$.
    Since $G$ is 3-connected and both $C$-contractions of $G$ are matching covered, both $C$-contractions of $G$ are isomorphic to $K_4$'s (up to multiple edges).
   It can be checked that $G$ is isomorphic to $\overline{C_6} $ or one of the graphs in Figure \ref{fig:C6+e-1}. By Proposition \ref{pro:H8_not_wheel-li},   $\overline{C_6}$ is  not  wheel-like.  It can be seen that all the graphs in Figure \ref{fig:C6+e-1} are not wheel-like  (the bold edges are removable).
    Therefore, the result holds.
\end{proof}
\begin{figure}[!h]
    \centering
    \includegraphics[totalheight=5.5cm]{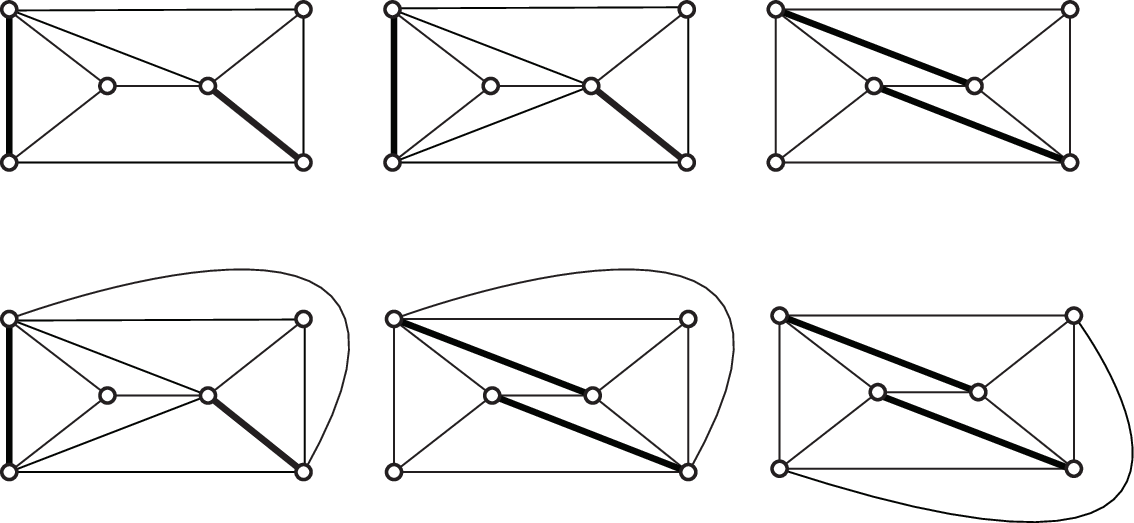}
    \caption{Planar bricks with six vertices, where the bold edges are removable.}
    \label{fig:C6+e-1}
\end{figure}

When we turn to the graphs with multiple edges, we have the following proposition.

\begin{pro}\label{lem:mutiedges}
    Let $G$ be a matching covered graph and $e\in E(G)$. And let $H$ be obtained from $G$ by adding an edge $f$ to two ends of $e$. Then $e$ and $f$ are removable in $H$.
    Therefore, if $G$ is not wheel-like, then $H$ is not wheel-like.
    \end{pro}
\begin{proof}
   As $H-e\cong G\cong H-f$ and $G$ is matching covered, both $e$ and $f$ are removable in $H$. Obviously, an edge that is removable in
   $G$ is removable in $H$.
   If $\{e,e'\}$ is a removable doubleton of $G$, then every perfect matching of $H$ containing $e$ or $f$ contains $e'$. So $e$ and  $f$ are forbidden in $H-e'$, that is,   $e'$ is not removable in $H$. Therefore, if $G$ does not contain any vertex $v$ such that    each removable class of $G$ has an edge in $\partial(v)$, neither does $H$.
\end{proof}

\begin{lem}\label{pro:mutiedges-add}
    Let $G$ be a wheel-like brick and $u_h$ is its hub. Then all the multiple edges are incident with $u_h$.
\end{lem}
    \begin{proof} Assume that the underlying simple graph of $G$ is $H$. If $G=H$, the result holds obviously. So we assume that $G\neq H$. If $H$ is near-bipartite, then $H$ is $K_4$ by Lemma \ref{lem:ne-bi_k4}. By Proposition \ref{lem:mutiedges}, all the multiple edges of $H$ are incident with a common vertex.  If $H$ is not near-bipartite, then every removable class of $H$ is a removable edge. By Theorem \ref{thm:re_in_brick}, $u_h$ is incident with at least three  removable edges. By Proposition \ref{lem:mutiedges} again, every multiple edge is incident with $u_h$.
\end{proof}

\begin{lem}\label{lem:at_least_one_wheel-like}
    Assume that $G_1$ and $G_2$ are two disjoint bricks,   $u\in V(G_1)$ and $v\in V(G_2)$.
    Let $G=({G_1(u)\odot G_2(v)})_{\theta}$. If $G$ is a wheel-like  brick, then the following statements hold. \\
    {\rm 1)} At least one of $G_1$ and $G_2$ is wheel-like such that $u$ or $v$ is its hub.\\
    {\rm 2)} If $G_1$ is  wheel-like,  $u$  is its hub, and every edge of $\partial_{G_1}(u)$ lies in some removable class of $G_1$, then $G_2$ is also wheel-like.
\end{lem}
\begin{proof} 1)
    By Theorem \ref{thm:re_in_brick}, $G_1$ and $G_2$ contain at least three removable classes, respectively.
    Suppose that   $R_1$ and $R_2$ are removable classes of $G_1$ and $G_2$, respectively, such that $R_1\cap \partial(u)=\emptyset$ and $R_2\cap \partial(v)=\emptyset$.
    Then $R_1\cup R_2$ is a matching in $G$.
    For $i\in \{1,2\}$, if $R_i$ is a removable edge, then $R_i$ is also a removable edge in $G$ by Lemma \ref{thm:re_also_re}, and if $R_i$ is a removable doubleton, then one edge of $R_i$ is also removable in $G$ by Theorem \ref{thm:removable doubleton}. It means that $R_1\cup R_2$ contains two nonadjacent removable edges of $G$,  contradicting the fact that $G$ is wheel-like.
    So at least one of $G_1$ and $G_2$ is wheel-like, such that $u$ or $v$ is its hub.

    2) Suppose, to the contrary, that  $G_2$ is not a wheel-like brick.
  Let $F_1$, $F_2$ and $F_3$ be removable classes of $G_2$ such that there does not exist any vertex in $G_2$ incident with one edge in $F_1$, $F_2$ and $F_3$, respectively. For $i=1,2,3$,
 if  $F_i\cap \partial(v)=\emptyset$, or  $F_i\cap \partial(v)=\{e_i\}$ and $\theta(e_i)$ is removable in $G_1$, then $F_i$ contains a  removable edge in $G$ by Lemma \ref{thm:re_also_re} and Theorem \ref{thm:removable doubleton};
 if  $F_i\cap \partial(v)=\{e_i\}$, and $\{\theta(e_i), e_i'\}$ is a removable doubleton in $G_1$, then by Theorem \ref{thm:removable doubleton}, the edge $e_i'$ is a removable edge in $G$. (If $e_i\in \partial(v)$,  $e_i$ lies in a removable doubleton of $G_2$, and  $\theta(e_i)$   lies in a removable doubleton in $G_1$, then $G$ has a removable doubleton by Lemma \ref{lem:removable doubleton}, and so the underlying simple graph of $G$ is $K_4$ by Lemma \ref{lem:ne-bi_k4}, contradicting the assumption that $G$ is  a splicing of two bricks.)  So,  in above both alternatives, for each $F_i$ ($i=1,2,3$), we have a removable class $F_i'$ of $G$ such that $F_i'\subset F_i$ or $F_i'\subset (E(G_1)\setminus \partial(u))$.
 Therefore, there does not exist any vertex  in $G$ incident with one edge in $F_1'$, $F_2'$ and $F_3'$, respectively, contradicting the fact that $G$ is wheel-like.
\end{proof}


\begin{lem}\label{Wi_Wj}
    Let $G$ and $H$ be two  odd wheels such that  $V(G)=\{u_h,u_1,u_2,\ldots, u_s\}$ and $V(H)=\{v_h,v_1,v_2,\ldots, v_t\}$, where $u_h$ and $v_h$ are the hubs of $G$ and  $H$, respectively. Assume that
     $u\in V(G)$, $v\in V(H)$,   $d_G(u)=d_H(v)$ and  $G(u)\odot H(v)$ is a brick.
    The graph $G(u)\odot H(v)$ is wheel-like if and only if the following statements hold.\\

    {\rm 1)}.  $|\{u,v\}\cap \{u_h,v_h\}|=1$. Without loss of generality, assume that $u=u_h$, that is $v \neq v_h$. Then $|V(G)|\ge 6$.

   {\rm 2)}. All the multiple edges of $G$ and $H$ are incident with $u_h$ and $v_h$, respectively.


    {\rm 3)}. Without loss of generality, assume that $v=v_t$  and $\{u_1v_1,u_rv_{t-1}\}\subset E(G(u)\odot H(v))$, where $1\leq r\leq s$. Then $r\neq 1$ and $u_1u_r\notin E(G)$.
\end{lem}
\begin{proof}
     Assume that $u_iu_{i+1}\in E(G)$ for $i=1,2,\ldots,s$ (the subscript is modulo $s$), and $v_iv_{i+1}\in E(H)$ for $i=1,2,\ldots,t$ (the subscript is modulo $t$).
   Let $W=G(u)\odot H(v)$ and $C=\partial_W( V( G)\setminus \{u\})$.

    We will prove the sufficiency firstly. By Statement 2), every multiple edge is incident with the hub.  Then $|E[v,v_h]|=d_G(u)-2$. We will prove the edge in $E(W)\setminus \partial(v_h)$ is nonremovable.
   As $u_1u_r\notin E(W)$, that is $r\neq 2$ and $r\neq s$, we consider the following alternatives.

  \textbf{Case 1.} $r=3~ \mbox{or}~ s-1$.\\

      Assume that $u_r=u_3$ (the case when $u_r=u_{s-1}$ is the same by symmetry). Let $B_{u_1u_2}=\{v_h,u_3,v_{t-2}\}$.
      Then $W-B_{u_1u_2}-{u_1u_2}$ has exactly three components: $W[\{u_2\}],W[\{v_{t-1}\}]$ and $W[V(G)\setminus (\{u_2,v_{t-1}\}\cup B_{u_1u_2}) ]$,
    all of  which are odd. Therefore, $B_{u_1u_2}$ is a barrier of $W-{u_1u_2}$ containing both ends of the edge $v_{t-2}v_h$. By Lemma \ref{forbidden}, $u_1u_2$ is nonremovable in $W$.
     By symmetry, $u_2u_3$ is also nonremovable in $W$.

     Assume that $s=5$. It can be checked $\{v_h,u_3,v_{t-2}\}$  is a barrier of $W-u_4u_5 $ containing $v_{t-2}v_h$.
      So $u_4u_5$ is nonremovable in $W$.
      If  $t=3$,   $\{v_h,u_1,u_{s-1}\}$ is a barrier of $W-v_1v_2 $ containing $u_{s-1}v_h$.
       So $v_1v_2$ are nonremovable in $W$.
     Every edge   of $E(W)\setminus \partial(v_h)\setminus \{u_1u_2,u_2u_3,u_4u_5,v_1v_2\}$ satisfies the triangle condition, and then it is not removable by Proposition \ref{pro:uvnonremovable}.
     If $t\ge5$, then   every edge in $E(W)\setminus \partial(v_h)\setminus \{u_1u_2,u_2u_3,u_4u_5\}$ satisfies the triangle condition; so it is not removable by Proposition \ref{pro:uvnonremovable} again.

     Now we assume that $s\ge7$. If  $t=3$, it can be checked   $\{v_h,u_1,u_{s-1}\}$ is a barrier of $W-v_1v_2 $ containing $u_{s-1}v_h$.
     So $v_1v_2$ are nonremovable in $W$.
     Every edge   of $E(W)\setminus \partial(v_h)\setminus \{u_1u_2,u_2u_3,v_1v_2\}$ satisfies the triangle condition, and then it is not removable by Proposition \ref{pro:uvnonremovable}.
     If $t\ge5$, then   every edge of $E(W)\setminus \partial(v_h)\setminus \{u_1u_2,u_2u_3\}$ satisfies the triangle condition; so it is not removable by Proposition \ref{pro:uvnonremovable} once more.

    \textbf{Case 2.} $r=4~ \mbox{or}~ s-2$.\\

     Without loss of generality, assume that $u_r=u_4$ (the case when $u_r=u_{s-2}$ is the same by symmetry).
     Let $B_{u_2u_3}=\{v_h,u_4,v_{t-2}\}$.
     Then $W-B_{u_2u_3}-{u_2u_3}$ has exactly three components: $W[\{u_3\}],W[\{v_{t-1}\}]$ and $W[V(G)\setminus (\{u_3,v_{t-1}\}\cup B_{u_2u_3}) ]$, which are odd.
     Therefore, $B_{u_2u_3}$ is a barrier of $W-{u_2u_3}$ containing both ends of the edge $v_{t-2}v_h$. By Lemma \ref{forbidden}, $u_2u_3$ is nonremovable in $W$. If $t\ge5$, then   every edge of $E(W)\setminus \partial(v_h)\setminus \{u_2u_3\}$ satisfies the triangle condition; so it is not removable by Proposition \ref{pro:uvnonremovable}.
     Then we consider the case when $t=3$.
     It can be checked $\{v_h,u_1,u_{s-1}\}$ is a barrier of $W-v_1v_2 $ containing $u_{s-1}v_h$.
     So $v_1v_2$ is nonremovable in $W$.
     Note that every edge of $E(W)\setminus \partial(v_h)\setminus \{u_2u_3,v_1v_2\}$ satisfies the triangle condition, and then it is not removable by Proposition \ref{pro:uvnonremovable} again.

    \textbf{Case 3.} $5\leq r\leq s-3$.\\

     If $t=3$,
     it can be checked $\{v_h,u_1,u_{s-1}\}$ is a barrier of $W-v_1v_2 $ containing $u_{s-1}v_h$.
     So $v_1v_2$ is nonremovable in $W$.
     Note that every edge   of $E(W)\setminus \partial(v_h)\setminus \{v_1v_2\}$ satisfies the triangle condition, and then it is not removable by Proposition \ref{pro:uvnonremovable}.
     If $t\ge5$, then every edge of $E(W)\setminus \partial(v_h)$ satisfies the triangle condition; so it is not removable by Proposition \ref{pro:uvnonremovable} again.
     Therefore, $W$ is wheel-like.

    We now turn to the necessary.
    By Lemma \ref{lem:at_least_one_wheel-like}, at least one of $G$ and $H$ is wheel-like such that $u$ or $v$ is its hub.
    Without loss of generality, assume that $G$ is wheel-like and $u=u_h$ (in fact, $G$ is an odd wheel).
    Since every edge in $\partial_G(u_h)$ is a removable class in $G$, $H$ is wheel-like by Lemma \ref{lem:at_least_one_wheel-like}.
    By Lemma \ref{pro:mutiedges-add}, every multiple edge of $G$ or $H$  is incident with the hubs, that is, Statement 2)  holds.
    We will obtain a contradiction that $W$ is not wheel-like if it does not satisfy Statements 1)-3).

    \textbf{Case A.} $v=v_i$  for $i\in\{1,2,\ldots,t\}$.

    Without loss of generality, assume that $v_i=v_t$ and $u_1v_1\in E(W)$. Assume that $u_{r}v_{t-1}\in E(W)$.
     As $W$ is a brick, it is 3-connected. Then $u_r\neq u_1$. Otherwise, $\{u_1,v_h\}$ is 2-vertex cut in $W$.

    If $s=t=3$, then the underlying simple graph of $W$ is isomorphic to $\overline{C_6}$  or the first two  graphs in Figure \ref{fig:C6+e-1}.
    By Proposition \ref{pro:H8_not_wheel-li}, $\overline{C_6}$  is not wheel-like.
   Note that  the first two graphs in Figure \ref{fig:C6+e-1} are not wheel-like. So the underlying simple graph of $W$ is not wheel-like. Therefore, $W$ is not wheel-like by Proposition \ref{lem:mutiedges},  a contradiction.

    Next, we consider $s=3$ and $t\ge5$.
    Since $(\partial(v_h)\setminus E[v_t,v_h])\cap C=\emptyset$ and every edge in $\partial(v_h)$ is removable in $H$, every edge in $\partial(v_h)\setminus E[v_t,v_h]$ is removable in $W$ by Lemma \ref{thm:re_also_re}.
       Without loss of generality,
    assume that   $r=2$. 
    Then it can be checked that  $W-u_1u_2$ can be gotten from an odd wheel by inserting a vertex into two edges on the rim of the odd wheel, respectively, up to multiple edges (incident with the hub).
    So $W-u_1u_2$  is matching covered, and then $u_1u_2$ is a removable edge in $W$.
    Then $u_1u_2$ and an edge of $\partial(v_h)\setminus E[v_t,v_h]$ are two nonadjacent removable edges in $W$, a contradiction.

    Now  we consider $s\ge5$ and $t\ge 3$.
       Suppose that $r=2~\mbox{or} ~s$.
    Similar to last paragraph (when $s=3$ and $t\ge5$), we can show that $u_1u_r$ is removable.
    Therefore, $u_1u_r$ and an edge of $\partial(v_h)\setminus E[v_t,v_h]$ are two nonadjacent removable edges in $W$, a contradiction. Therefore, $2<r<s$, that is, Statements 1) and 3) hold.


    \textbf{Case B.} $v=v_h$.

    If $s\ge 5$ and $t\ge5$, then every edge in $\partial(u_h)$ and $\partial(v_h)$ is removable in $G$ and $H$, respectively,
    and hence the  edges in $C$
    are removable in $W$ by Lemma \ref{thm:re_also_re}.
    As $s\ge 5$ and $t\ge5$, there exist two nonadjacent edges in $C$, contradicting the fact that $W$ is wheel-like.


    Assume that $s=3$ and $t\ge5$ (the case when $t=3$ and $s\ge5$ is the same by exchanging the roles of $G$ and $H$).
    If exactly one vertex  in $\{u_1,u_2,u_3\}$, say $u_1$, satisfies $|E[u_h,u_1]|\ge2$,  then, similar to the case when $t=3$ and $s\ge5$ in Case A (by exchanging the roles of $G$ and $H$), we have $W$ satisfies Statements 1)-3).
    If  there exist multiple edges between $u_h$ and at least two different vertices in $\{u_1,u_2,u_3\}$,  then this case is similar to last paragraph (when $s\ge5$ and $t\ge5$ in Case B).
    So there exist two nonadjacent removable edges  in $W$, a contradiction.

    Last, we consider the case when $s=t=3$.
    Then the underlying simple graph of $W$ is isomorphic to $\overline{C_6}$,  or one of  graphs in Figures \ref{fig:C6+e-1} and  \ref{fig:case-b-1}.
    Recall that  $\overline{C_6}$  is not wheel-like.
    It can be checked  that all the  graphs in Figures \ref{fig:C6+e-1} and \ref{fig:case-b-1} are not wheel-like. So the underlying simple graph of $W$ is not wheel-like. Therefore, $W$ is not wheel-like by Proposition \ref{lem:mutiedges},
    a contradiction.
\end{proof}
 \begin{figure}[!h]
    \centering
    \includegraphics[totalheight=5.4cm]{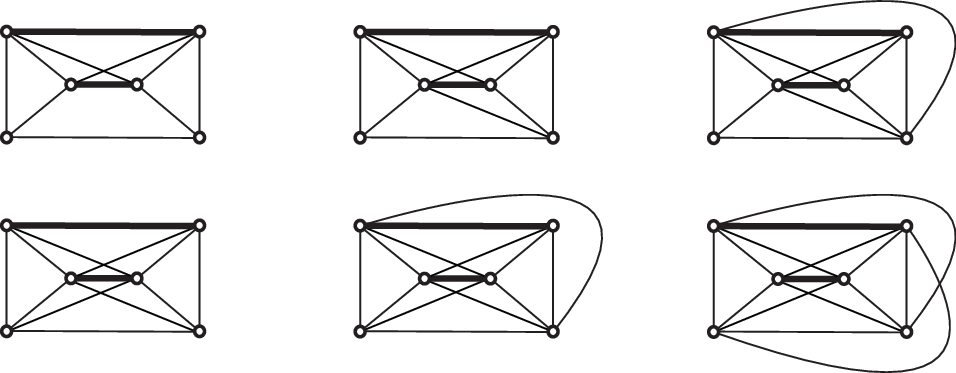}
    \caption{Illustration for the proof of Lemma \ref{Wi_Wj}, where  the bold edges are removable.}
    \label{fig:case-b-1}
\end{figure}
\begin{lem}\label{nonplanar}
     Assume that  $W$  is the wheel-like brick which is the splicing   of two odd wheels.
     Then $W$ is nonplanar.
\end{lem}
\begin{proof}
      By Lemma \ref{Wi_Wj}, $W$ satisfies Statements 1)-3) of Lemma \ref{Wi_Wj}. Using the notations in the proof of Lemma \ref{Wi_Wj},
     $W-\{u_1,u_r,v_h\}$ has three components $C_1=W[\{u_2,\ldots, u_{r-1}\}]$, $C_2=W[\{u_{r+1},\ldots, u_s\}]$ and $C_3=W[\{v_1,\ldots, v_{t-1}\}]$.
   Then,  $C_i$ ($i\in\{1,2,3\}$), $u_1$,$u_r$,$v_h$  and suitable paths between them form a subdivision of $K_{3,3}$ (the vertex set $\{u_1,u_r,v_h\}$ is one of its color class).
    By Theorem \ref{thm:plaran_K3,3}, $W$ is nonplanar.
\end{proof}

\begin{lem}\label{bipar-H-nonre}
    Assume that $\partial(X)$ and $\partial(Y)$ are two  cuts of a brick $G$ such that  $G/{X}$ and $G/{Y}$ are bricks, and   $G/(\overline{X}\rightarrow\overline{x})/(\overline{Y}\rightarrow\overline{y})$ is a matching covered bipartite graph $H$.
    Then every edge incident with $\overline{x}$ is removable in $H$.
\end{lem}
\begin{proof}
    Let $A$ and $B$ be the two color classes of $H$. Without loss of generality, assume that $\overline{x}\in A$. Suppose, to the contrary, that  $\overline{x}b$ is nonremovable in $H$ ($b$ and $\overline{y}$ may be the same vertex). Then by Lemma \ref{lem:nonre-bi} there exists an edge cut $\partial(Z)$ in $H$, such that $\overline{x}\in Z\cap A$, $b\in\overline{Z}\cap B$, $|Z\cap A|=|Z\cap B|$ and $\{\overline{x}b\}=E[Z\cap A,\overline{Z}\cap B]$.

    As $\{\overline{x}b\}=E[Z\cap A,\overline{Z}\cap B]$, $N((\overline{Z}\cap B)\setminus\{b\})\subseteq\overline{Z}\cap A$.
    Let $S=\overline{Z}\cap A$ and $Q=Z\cup\{b\}$.
    As $G/{X}$ and $G/{Y}$ are bricks, $G/{X}$ and $G/{Y}$ are 3-connected. Therefore, both $G[\overline{X}]$ and $G[\overline{Y}]$ are connected and have odd numbers of  vertices, respectively.
    If $\overline{y}\notin Q$, then every vertex of $(\overline{Z}\cap B)\setminus\{b,\overline{y}\}$, $G[Q]$ and $G[\overline{Y}]$ are the components of $G-S$.
    If $\overline{y}\in Q$, then $|(Q\cup \overline{Y})\setminus \{\overline{y}\}|$ is also odd in $G$, and every vertex of $(\overline{Z}\cap B)\setminus\{b\}$ and $Q$ are the components of $G-S$. We have $o(G-S)=|S|$ whether $\overline{y}$ is in $Q$ or not, as $|\overline{Z}\cap A|=|\overline{Z}\cap B|$.
    So $S$ is a barrier of $G$. Then $ \partial (Q)$ is a tight cut by a simple
computation.
    Noting $G$ is a brick,  $G$ is free of nontrivial tight cuts.
    As $\{b, \overline{x}\}\subset Q$, we have $|\overline{Q}|=1$, that is, $|\overline{Z}\cap A|=1$. Since $|\overline{Z}\cap A|=|\overline{Z}\cap B|$, we have $|\overline{Z}\cap B|=1$, that is, $\{b\}=\overline{Z}\cap B$.

    Assume that $\overline{y}\neq b$. Note that $\overline{x}b$ is nonremovable in $H$.
    Recalling $|\overline{Z}\cap A|=1$, we have $|N_G(b)|=2$,   contradicting the assumption that $G$ is a brick.
    Now we consider the case when $\overline{y}=b$. It can be checked that $Z\cap B$ is also a barrier of $G$ by symmetry. So $|Z\cap B|=1$. Let $\{b'\}=Z\cap B$. Similar to the case when $\overline{y}\neq b$, we have $|N_H(b')|=2$,
     and so  $|V(H)|=4$.  In $G$, assume that $e$ is the corresponding edge  of the edge $\overline{x}\,\overline{y}$. Then the end of $e$ in  $\overline{X}$ and $b'$ is a 2-vertex cut of $G$, contracting the fact that $G$ is a brick.
      \end{proof}

\begin{lem}\label{lem:V(H)_re_edge}
    Assume that $\partial(X)$ and $\partial(Y)$ are two cuts of a brick $G$ such that  $G/{X}$ and $G/{Y}$ are bricks, and   $(G/(\overline{X}\rightarrow\overline{x}))/(\overline{Y}\rightarrow\overline{y})$ is a matching covered bipartite  graph $H$.
    If $G/(X\rightarrow x)$ is an odd wheel such that $x$ is its hub and $|N_H(\overline{x})|\ge 2$, then there exists a removable edge $e$ of $G$ such that both  ends of $e$ belong to $\overline{X}\cup N(\overline{X})\setminus\overline{Y}$.
\end{lem}
\begin{proof}
      Let $G'=G/(X\rightarrow x)$.
    As $x$ is the hub of $G'$, every edge of $\partial(x)$ lies in some removable class.
    Note that $|N_H(\overline{x})|\ge 2$.
    Assume that $\overline{x}b$ is an edge of $\partial(\overline{x})$ in $H$ where $b\neq \overline{y}$, and $\theta(\overline{x}b)$ is the edge in $G'$ corresponding to $\overline{x}b$.
    Then $\overline{x}b$ is removable in $H$ by Lemma \ref{bipar-H-nonre}.
    If $\theta(\overline{x}b)$ is removable in $G'$, then $\overline{x}b$ is also removable in $G/\overline{Y}$ by Lemma \ref{thm:re_also_re} as $G/\overline{Y}=G'(x)\odot H(\overline{x})$.
    Since $\overline{x}b\notin \partial(\overline{y})$,
    $\overline{x}b$ is removable in $G$ by Lemma \ref{thm:re_also_re} again.
    Moreover, the ends of $\overline{x}b$ belong to $\overline{X}\cup N(\overline{X})\setminus\overline{Y}$. The result follows by setting $e=\overline{x}b$ in this case.
    If $\theta(\overline{x}b)$ is an edge of a removable doubleton in $G'$,  assume that $\{e',\theta(\overline{x}b)\}$  is  a removable doubleton  in $G'$. Then $e'$ is a removable edge in $G/\overline{Y}$ by Theorem \ref{thm:removable doubleton}.
    So $e'$ is removable in $G$ by Lemma \ref{thm:re_also_re} once more. As the ends of $e'$ belong to $\overline{X}\cup N(\overline{X})\setminus\overline{Y}$, the result follows by setting $e=e'$.
    \end{proof}

\section{Proof of  Theorem \ref{main}}

Let $G$ be a  planar wheel-like brick.
By Lemma \ref{pro:mutiedges-add}, we only need to consider the case when $G$ contains no multiple edges.
If  $G$ is a  near-bipartite brick, then  $G$ is isomorphic to $K_4$ by Lemma \ref{lem:ne-bi_k4}.
Now we consider that   $G$ is not a near-bipartite brick.
We will prove this by induction on $|V(G)|$.
If $|V(G)|=6$, then  $G$ is $W_5$ by Proposition \ref{lem:six_vertices_W5}. The result follows.
So we assume that the result holds for all the bricks with less than $n(>6)$ vertices, where $n$ is even.
We now assume that $|V(G)|=n$.
If $G$ is a solid brick, then $G$ is an odd wheel by Theorem \ref{thm:planar_solid}. The theorem holds.
So we consider the case when $G$ is a nonsolid brick.
Then $G$ has a robust cut  by Theorem \ref{exist-robust}.
Let $\partial(Z)$ be a robust cut of $G$ such that  $G_1=G/(Z \rightarrow z)$ and $G_2=G/(\overline{Z}\rightarrow \overline{z})$ are near-bricks.
Since $G$ is planar, both $G_1$ and $G_2$ are planar by Lemma \ref{lem:planarG/X}.

We assume firstly that both of $G_1$ and $G_2$ are bricks.
Recall that $G$ is wheel-like.
Then at least one of $G_1$ and $G_2$, say $G_1$, is wheel-like and $z$ is its hub by Lemma \ref{lem:at_least_one_wheel-like}.
By the induction hypothesis,  $G_1$  is an odd wheel.
Then every edge of $\partial(z)$ lies in some   removable class of $G_1$.
By Lemma \ref{lem:at_least_one_wheel-like}, $G_2$ is wheel-like.
By the induction hypothesis, $G_2$ is an odd wheel.
By Lemmas \ref{Wi_Wj} and \ref{nonplanar}, the wheel-like brick $G$ is nonplanar, a contradiction.

Now we assume that at least one of $G_1$ and $G_2$ is not a brick. Then there are two  cuts, say $\partial(X)$ and $\partial(Y)$, of   $G$ such that $G/(\overline{X}\rightarrow\overline{x})/(\overline{Y}\rightarrow\overline{y})$ is a bipartite matching covered graph $H$, and $G'=G/(X \rightarrow x)$ and $G''=G/(Y \rightarrow y)$  are bricks
by Theorem \ref{robust_H}.

 \noindent
    \textbf{Claim 1.} $G'$ and $G''$ are wheel-like.
\begin{proof} Similar to the proof of 1) of Lemma \ref{lem:at_least_one_wheel-like} (with $G'$ and $G''$ in place of $G_1$ and $G_2$, respectively), we have at least one of $G'$ and $G''$ is wheel-like, such that $x$ or $y$ is its hub.

    Without loss of generality, we assume that $G'$ is wheel-like and $x$ is its hub.
    By Lemma \ref{lem:planarG/X}, $G'$ is planar.
    So $G'$ is an odd wheel by the induction hypothesis.
    Suppose that $G''$ is not wheel-like.
    Then there is a removable class $R_3$ in $G''$ such that $R_3\cap \partial(y)=\emptyset$.  
    By Lemma \ref{thm:re_also_re} and Theorem \ref{thm:removable doubleton}, $R_3$ contains a removable edge $e_0$ in $G$.
    Suppose that $|N_H(\overline{x})|\ge2$.
    By Lemma \ref{lem:V(H)_re_edge}, there exists a  removable  edge $e$ in $G$ such that both  ends of $e$ belong to $\overline{X}\cup N(\overline{X})\setminus\overline{Y}$.
    So $e$ and $e_0$ are two nonadjacent removable edges in $G$, contradicting the assumption that $G$ is wheel-like.
    Therefore, we have $|N_H(\overline{x})|=1$.
    As $H$ is matching covered, $|V(H)|=2$ and $V(H)=\{\overline{x},\overline{y}\}$.
    So $G$ is isomorphic to $G'(x)\odot G''(y)$.
    Recalling that $G'$ is an odd wheel, every edge of $\partial(x)$ lies in some removable class of $G'$.
    By Lemma \ref{lem:at_least_one_wheel-like}, $G''$ is wheel-like since $G$ is wheel-like.
\end{proof}

 By Lemma \ref{lem:planarG/X}, $G'$ and $G''$ are planar. So
  $G'$ and $G''$ are odd wheels by the induction hypothesis and  Claim 1. Assume that $u_h$ and $v_h$ are the hubs of $G'$ and $G''$, respectively.
    Let  $N_{G''}(v_h)=\{v_1,\ldots, v_{k}\}$.
    By the proof of Claim 1,
    without loss of generality,
    assume that $x=u_h$. Moreover, we have the following claim.

\noindent
    \textbf{Claim 2.} $V(H)=\{\overline{x},\overline{y}\}$.
\begin{proof}
    Suppose, to the contrary, that $|N_H(\overline{x})|\ge2$.
    By Lemma \ref{lem:V(H)_re_edge}, there exists a  removable edge $e_1$ in $G$  such that both  ends of $e_1$ belong to $\overline{X}\cup N(\overline{X})\setminus\overline{Y}$.
    If $y=v_i$ $(i\in\{1,\ldots,k\})$, then every edge of $\partial(v_h)\backslash E[v_i,v_h]$ is a removable class $R$ in $G''$.
    Note that $R\cap \partial(y)=\emptyset$.
    If $R$ is a removable edge, then $R$ is also a removable edge in $G$ by Lemma \ref{thm:re_also_re}, and if $R$ is a removable doubleton, then one edge of $R$ is also removable in $G$ by Theorem \ref{thm:removable doubleton}.
    So there exists an edge $e_2$ that is removable in $G$ and both  ends of $e_2$ belong to $\overline{Y}$.
    Therefore, $e_1$ and $e_2$ are two nonadjacent removable edges in $G$, contradicting the assumption that $G$ is wheel-like.
    So $y=v_h$.
    Suppose that $|N_H(\overline{y})|\ge2$. Then there exists an edge $e_3$ that is removable in $G$ and both  ends of $e_3$ belong to $\overline{Y}\cup N(\overline{Y})\setminus\overline{X}$ by Lemma \ref{lem:V(H)_re_edge}.
    So $e_1$ and $e_3$ are two nonadjacent removable edges in $G$, a contradiction.
    Therefore, we have $|N_H(\overline{y})|=1$.
    Since $H$ is matching covered, it is connected. Then $N_H(\overline{y})=\{\overline{x}\}$, that is $V(H)=\{\overline{x},\overline{y}\}$.\end{proof}

   By Claim 2,  $G'(x)\odot G''(y)$ is isomorphic to $G$.   Recall that $G'$ and $G''$ are odd wheels.
   By Lemma \ref{nonplanar},
   the graph $G$ is nonplanar, a contradiction. Therefore, the result holds.              





\end{document}